\documentclass[reqno,11pt,letterpaper]{amsart}

\usepackage{latexsym,fancyhdr}
\usepackage{amsmath,amssymb,xcolor}
\usepackage{booktabs}
\usepackage{multicol}
\usepackage{mathtools}
\usepackage{wrapfig}
\usepackage{times}
\usepackage{pb-diagram}
\usepackage[all,cmtip]{xy} 
\usepackage{lineno}
\usepackage{caption}
\usepackage{mathabx}%for \notequiv
\usepackage{tikz,url}
\usepackage{etoolbox} 
\usepackage{cleveref}
\usepackage{graphicx} % Required for inserting images

\usepackage{algorithm}
\usepackage[noend]{algpseudocode}

\newtheorem{Tma}{Theorem}[section]

\newtheorem{lemma}[Tma]{Lemma}
\newtheorem{conjecture}[Tma]{Conjecture}
\newtheorem{proposition}[Tma]{Proposition}
\newtheorem{corollary}[Tma]{Corollary}
\newtheorem{theorem}[Tma]{Theorem}

\theoremstyle{definition}

\newtheorem{remark}[Tma]{Remark}
\newtheorem{problem}[Tma]{Problem}
\newtheorem{compprop}[Tma]{Computational Proposition}

\definecolor{lblue}{RGB}{50,90,250}

\newcommand{\Aut}{\mathrm{Aut}}

\newcommand{\val}{\mathrm{val}}

\newcommand{\nonsplit}[2]{#1\raisebox{0.6ex}{$\cdot$} #2}

\title{An infinite family of counterexamples to the Polycirculant Conjecture}
\author{Saul D. Freedman and Melissa Lee}
\date{\today}

\address[Freedman]{Department of Mathematics, Colorado State University, Fort Collins, CO 80523, USA}
\address[Lee]{School of Mathematics, Monash University, Clayton, Australia}

\email{saul.freedman@colostate.edu, melissa.lee@monash.edu}

\thanks{The authors would like to thank the Isaac Newton Institute for Mathematical Sciences, Cambridge, for support and hospitality during the programme \emph{Algebraic groups, geometry, invariants and related topics}, where work on this paper was undertaken. This work was supported by EPSRC grant EP/Z000580/1.}

\begin{document}

\begin{abstract}
We disprove the Polycirculant Conjecture, which states that every transitive $2$-closed permutation group is non-elusive, i.e.~contains a derangement of prime order. In fact, we strengthen our result by answering a long-standing question of Maru{\v{s}}i{\v{c}} and Jordan: there exists a vertex-transitive graph admitting no semiregular automorphism. To do so, we generalise recently developed methods of Chen et al.~for constructing elusive groups via non-split extensions, allowing us to construct an elusive group $\nonsplit{7^6}{\mathrm{PSU}_3(3)}$ of degree $16{,}464$. We show that this group is the full automorphism group of seven of its orbital graphs and hence is $2$-closed. Our example extends to infinitely many counterexamples of the Polycirculant Conjecture, and infinitely many vertex-transitive graphs admitting no semiregular automorphism.
\end{abstract}

\maketitle

\section{Introduction}
\label{sec:intro}

A \textit{derangement} is a finite permutation that has no fixed points. A classical theorem of Jordan from the nineteenth century asserts that every finite transitive group contains a derangement.  Using the Classification of Finite Simple Groups, Fein, Kantor and Schacher \cite{fein1981relative} showed that Jordan's theorem may be strengthened to require the existence of a derangement of prime power order. However, perhaps surprisingly, it is false that every finite transitive permutation group must contain a derangement of prime order. For instance, the Mathieu group $M_{11}$ acting on 12 points has derangements only of orders 4 and 8. Due to the seeming rarity of finite transitive permutation groups containing no derangements of prime order, such a group is termed \textit{elusive}.

Elusive groups have attracted significant interest (see \cite{MR4207924,MR3443032,burness2018locally,MR2824515,cameron2002transitive,giudici2009characterizing,MR4400061,MR4793211} for example), in part due to their connection to the well-known Polycirculant Conjecture, posed by Klin \cite[Problem BCC15.12]{cameron1997problems} in 1997. Recall that the \emph{$2$-closure} of a permutation group $G\leq \mathrm{Sym}(\Omega)$ is the largest subgroup of $\mathrm{Sym}(\Omega)$ with the same orbits on $\Omega\times \Omega$ as $G$, and that $G$ is \emph{$2$-closed} if it is its own $2$-closure.

\begin{conjecture}[The Polycirculant Conjecture]
No elusive group is $2$-closed.
\end{conjecture}

This conjecture extends the earlier Semiregularity Problem, formulated by Maru{\v{s}}i{\v{c}} \cite{maruvsivc1981vertex} and Jordan \cite{jordan1988symmetrieeigenschaft} in the 1980s. A non-trivial permutation is called \emph{semiregular} if all of its cycles are the same length (so in particular, it has no trivial cycles).

\begin{problem}[The Semiregularity Problem]
Does there exist a vertex-transitive directed graph admitting no semiregular automorphism?
\end{problem}

This problem includes the case where the arc set of the directed graph is symmetric, so that the graph can be considered undirected. Observe that each derangement of prime order is semiregular, and each semiregular permutation has a power equal to a derangement of prime order. Hence the Semiregularity Problem asks if there is a vertex-transitive directed graph whose full automorphism group is elusive. Note that, by definition, the automorphism group of any finite directed graph is $2$-closed. Hence a positive answer to the Semiregularity Problem would imply that the Polycirculant Conjecture is false. On the other hand, there exist $2$-closed groups that are not the full automorphism groups of any directed graph, such as the subgroup $C_2 \times C_2$ of $S_4$; additional examples are given in \cite{babai80,bgs24,gmz23}. Therefore, a counterexample to the Polycirculant Conjecture might not provide an answer to the Semiregularity Problem.

The Polycirculant Conjecture and Semiregularity Problem have received considerable attention across dozens of papers; see \cite{arezoomand2019problems} for a survey of results. Perhaps most notably, Giudici \cite{giudici2003quasiprimitive} determined all elusive permutation groups with at least one transitive minimal normal subgroup (hence all elusive quasiprimitive groups) and showed that each satisfies the Polycirculant Conjecture. A few years later, Giudici and Xu \cite{giudicixu} proved that the conjecture is satisfied for every group that is \emph{biquasiprimitive}, i.e.~transitive with each minimal normal subgroup having at most two orbits. Hence in any counterexample to the conjecture, each minimal normal subgroup must be intransitive, and at least one must have at least three orbits.

Recently, Chen et al.~\cite{chen2026elusive} described a novel method of constructing elusive groups via non-split extensions, and used the method to construct the first known examples of elusive groups with odd and twice-odd degree. Their examples again satisfy the Polycirculant Conjecture.

In this paper, we further the methods of Chen et al.~to construct the first known groups that are counterexamples to the Polycirculant Conjecture, and the first known graphs that are positive answers to the Semiregularity Problem. In fact, we construct infinitely many such groups and graphs.

To state our main results, we require the following notation. We will write $\nonsplit{N}{G}$ to denote a non-split extension of a group $N$ by a group $G$, with $N \mathbin{:} G$ denoting a split extension of $N$ by $G$ (not to be confused with the index $|G:H|$ in $G$ of a subgroup $H$, nor the set $[G:H]$ of right cosets of $H$ in $G$). Additionally, for a prime $p$ and a positive integer $a$, we write $p^a$ to denote the elementary abelian group of order $p^a$. Recall also that an \emph{orbital} $O$ of a permutation group $G$ on a set $\Omega$ is an orbit of $G$ on $\Omega \times \Omega$, and that the corresponding \emph{orbital graph} $\Gamma_O$ is the directed graph with vertex set $\Omega$ and arc set $O$. Clearly, each element $g \in G$ is an automorphism of $\Gamma_O$, with $(\alpha,\beta)^g = (\alpha^g,\beta^g)$ for each $(\alpha,\beta) \in O$. Hence $\Gamma_O$ is edge-transitive, as well as vertex-transitive if $G$ is transitive. If $O$ is \emph{self-paired}, i.e.~if $(\beta,\alpha) \in O$ for all $\alpha, \beta \in \Omega$ with $(\alpha,\beta) \in O$, then we can consider the self-paired orbital graph $\Gamma_O$ as an undirected graph.

\begin{theorem}
\label{thm:main}
There is an elusive permutation group $X$ of degree $16{,}464 = 2^4.3.7^3$ isomorphic to $\nonsplit{7^6}{\mathrm{PSU}_3(3)}$, with point stabiliser of shape $7^4\mathbin{:}(C_3\times S_3)$, such that $X$ has (at least) seven pairwise non-isomorphic, connected, self-paired orbital graphs $\Gamma_O$ of valency $63$, with $\Aut(\Gamma_O) = X$ for each $\Gamma_O$. Therefore, each of these vertex-transitive graphs of order $16{,}464$ admits no semiregular automorphism.
\end{theorem}

These graphs also provide a negative answer to \cite[Question 1.1]{gpv14}, which can be stated slightly more generally as follows: does every connected, regular, edge-transitive graph admit a semiregular automorphism?

\begin{remark}
\label{rem:hamiltonian}
A well-known open problem of Lovász \cite{lovaszproc} asks whether every vertex-transitive graph has a Hamiltonian path. This problem is often stated as a conjecture in the affirmative, and it has been similarly conjectured (see \cite[pp.~162--163]{bermond} and \cite[p.~46]{godsilroyle}) that, apart from five known counterexamples, every vertex-transitive graph has a Hamiltonian cycle. On the other hand, Babai \cite{babai} conjectured that infinitely many counterexamples exist. As certain methods of constructing Hamiltonian paths and cycles rely on the existence of semiregular automorphisms (see \cite{alspach}), it is natural to test the graphs $\Gamma_O$ from Theorem~\ref{thm:main} against these conjectures, especially since the valency of each $\Gamma_O$ is small compared to its order. However, straightforward computations in SageMath \cite{sagemath} show that each of these graphs has a Hamiltonian cycle.
\end{remark}

Theorem~\ref{thm:main}, together with the elementary result Proposition~\ref{prop:dir_prod} on direct products and wreath products of $2$-closed groups and elusive groups, immediately yields the following corollary.

\begin{corollary}
\label{cor:2closed}
The elusive permutation group $X$ is $2$-closed. Moreover, for each integer $k \ge 1$, any group constructed from $k$ copies of $X$ via direct products and/or wreath products is isomorphic to a $2$-closed elusive group of degree $16{,}464^k$. These groups are therefore all counterexamples to the Polycirculant Conjecture.
\end{corollary}

As discussed in Section~\ref{sec:nonsplit}, this corollary gives a positive answer to \cite[Problem 5.2]{chen2026elusive} on counterexamples to the Polycirculant Conjecture constructed via non-split extensions with quasiprimitive quotients. 

\begin{remark}
\label{rem:otheractions}
The group $X$ in fact has a second faithful transitive elusive action of degree $16{,}464$ (the point stabiliser in the action from Theorem~\ref{thm:main} has exactly 16 normal subgroups, while the point stabiliser in the second action, which also has shape $7^4\mathbin{:}(C_3\times S_3)$, has exactly nine normal subgroups). However, in this second action, the hypotheses of Theorem 5.3 in \cite{cameron2002transitive} apply to the socle $7^6$ of $X$. Thus this theorem implies that the $2$-closure of this socle, which lies in the $2$-closure of $X$, contains a derangement of prime order. Therefore, this action of $X$ satisfies the Polycirculant Conjecture. Furthermore, writing $\mathrm{P}\Gamma\mathrm{U}_3(3)$ to denote the automorphism group $\mathrm{PSU}_3(3)\mathbin{:}2$ of $\mathrm{PSU}_3(3)$, there is a unique elusive permutation group $U$ of degree $16{,}464$ isomorphic to $\nonsplit{7^6}{\mathrm{P}\Gamma\mathrm{U}_3(3)}$ (with $\mathrm{P}\Gamma\mathrm{U}_3(3)$ acting irreducibly on $7^6$). This group contains a copy of $X$ in its second elusive action, and so $U$ again satisfies the Polycirculant Conjecture.
\end{remark}

We also use a graph product construction to extend the graphs from Theorem~\ref{thm:main} to infinite families of graphs that give a positive answer to the Semiregularity Problem.
\begin{corollary}
\label{cor:graphs}
There exist countably many infinite families of graphs of orders $16{,}464^k$ for each $k\geq 1$ that are vertex-transitive and do not admit semiregular automorphisms.
\end{corollary}

As mentioned above, any counterexample to the Polycirculant Conjecture must have each minimal normal subgroup intransitive, with at least one having at least three orbits. Our method of construction of the counterexample $X$ to the conjecture involves the fact that the elementary abelian normal subgroup $N = 7^6$ of $X$ is an irreducible $\mathrm{PSU}_3(3)$-module. Since $\mathrm{PSU}_3(3)$ is simple and $X$ is non-split, this implies that $N$ is the unique minimal normal subgroup of $X$. We note that $N$ has $336$ orbits, and that $X$ has order $711{,}541{,}152 = 2^5.3^3.7^7$ (and degree $16{,}464 = 2^4.3.7^3$, as mentioned above). Additionally, the graphs $\Gamma_O$ from Theorem~\ref{thm:main} each have valency $63$ (and order $16{,}464$, the degree of $G$). Recall also that there exist $2$-closed groups that are not the full automorphism groups of any directed graphs. The following open problems therefore naturally arise. Given a permutation group $G$, we write $\xi(G)$ to denote the maximum number of orbits of $M$ among all minimal normal subgroups $M$ of $G$.

\begin{problem}
\label{prob:min}
Among all counterexamples $G$ to the Polycirculant Conjecture, determine:
\begin{enumerate}
\item  the smallest degree of $G$; 
\item the smallest order of $G$; 
\item the least number of orbits in a minimal normal subgroup of $G$; and 
\item the smallest value of $\xi(G)$.
\end{enumerate} Similarly, among all vertex-transitive directed graphs $\Gamma$ admitting no semiregular automorphism, determine:
\begin{enumerate}
    \item[(5)]  the smallest order of $\Gamma$; and
    \item[(6)] the smallest valency (or out-valency) of $\Gamma$.
\end{enumerate}
\end{problem}

Note that the aforementioned survey paper \cite{arezoomand2019problems} includes results on necessary conditions on the valency and other properties of such a graph $\Gamma$.

\begin{problem}
Does there exist a counterexample $G$ to the Polycirculant Conjecture, such that $G$ is not the full automorphism group of a vertex-transitive directed graph?
\end{problem}

If no such $G$ exists, then Parts (1) and (5) of Problem~\ref{prob:min} have the same answer.

The remainder of the paper is organised as follows. Section \ref{sec:prelims} outlines some preliminary results, and we discuss constructions of elusive groups from non-split extensions in Section~\ref{sec:nonsplit}. The proofs of Theorem~\ref{thm:main} and Corollary~\ref{cor:graphs} are then detailed in Section \ref{sec:prf}. Section~\ref{sec:prf} also includes a proof of Corollary~\ref{cor:2closed}, combining theoretical and computational techniques, that does not rely on Theorem~\ref{thm:main}. The techniques in this proof may be useful when considering further potential counterexamples to the Polycirculant Conjecture that are not amenable to the methods used to prove Theorem~\ref{thm:main}, in particular those that may not be the full automorphism group of any vertex-transitive directed graph.

\section{Preliminaries}
\label{sec:prelims}
We first introduce some notation. For a permutation group $G \leq \mathrm{Sym}(\Omega)$, and $\Sigma$ a $G$-invariant partition of $G$ on $\Omega$, we write $G^\Sigma$ to denote the subgroup of $\mathrm{Sym}(\Sigma)$ induced by $G$ on $\Sigma$. For $\Delta \in \Sigma$, we write $G_\Delta$ for the setwise stabiliser of $\Delta$ in $G$, and $G_\Delta^\Delta$ for the subgroup of $\mathrm{Sym}(\Delta)$ induced by $G_\Delta$ on $\Delta$. The $2$-closure of $G$ is denoted by $\kappa(G)$, where the action is clear from context.

Our proof of Corollary \ref{cor:2closed} requires the following structural results.

\begin{lemma}[{{\cite[Lemma 3.1]{totally2closed}}}]
\label{lem:part}
Let $G\leq \mathrm{Sym}(\Omega)$ be a transitive group, and suppose that $\Sigma$ is a non-trivial $G$-invariant partition of $\Omega$. Then the following statements hold.
\begin{enumerate}
    \item The partition $\Sigma$ is also $\kappa(G)$-invariant.
    \item The induced group $\kappa(G)^\Sigma$ is a subgroup of $\kappa(G^\Sigma)$.
    \item For $\Delta\in \Sigma$, the group $\kappa(G)_\Delta^\Delta$ is a subgroup of $\kappa(G_\Delta^\Delta)$.
\end{enumerate}
\end{lemma}

\begin{theorem}[Imprimitive Wreath Embedding Theorem]
\label{thm:imprim_wr}
Let $G\leq \mathrm{Sym}(\Omega)$, let $\Sigma$ be a $G$-invariant partition of $\Omega$ with finitely many blocks, and let $\Delta \in \Sigma$. Further assume that the group $G^\Sigma$ induced by $G$ on $\Sigma$ is transitive. Then $G$ is permutationally isomorphic to a subgroup of $G_\Delta^\Delta \wr G^\Sigma$ acting on $\Delta \times \Sigma$.
\end{theorem}

See {\cite[Theorem 5.5]{praegerschneider}} for a proof, which details a permutational isomorphism that can be used to embed the group $G$ in the specified wreath product. A more direct description of the permutational embedding of $G$ is given in \cite[Section II.4]{hulpkenotes}.

Our next result is the case $k = 2$ of the theorem {\cite[Theorem 5.6]{wielandt}} on \emph{$k$-closures} of permutation groups for each $k \ge 1$.
\begin{proposition}
\label{prop:wielandt}
Let $G\leq \mathrm{Sym}(\Omega)$. A permutation $\sigma \in \mathrm{Sym}(\Omega)$ lies in $\kappa(G)$ if and only if for all $\alpha, \beta \in \Omega$, there exists $g\in G$ with $\alpha^\sigma=\alpha^g$ and $\beta^\sigma=\beta^g$. 
\end{proposition}

The following proposition shows that any $2$-closed elusive group gives rise to an infinite family of $2$-closed elusive groups.
\begin{proposition}
\label{prop:dir_prod}
Let $G_1\leq \mathrm{Sym}(\Omega_1)$ and $G_2 \leq \mathrm{Sym}(\Omega_2)$ be  transitive permutation groups.
\begin{enumerate}
    \item \label{i1} If $G_1$ and $G_2$ are elusive, then $G_1\times G_2$ is elusive in its (product) action on $\Omega_1 \times \Omega_2$, and $G_1 \wr G_2$ is elusive in its (imprimitive) action on $\Omega_1 \times \Omega_2$.
        \item \label{i2} If $G_1$ and $G_2$ are $2$-closed, then $G_1\times G_2$ and $G_1 \wr G_2$ are $2$-closed in their actions on $\Omega_1 \times \Omega_2$.
\end{enumerate}
\end{proposition}
\begin{proof}
Part (\ref{i1}) is \cite[Theorem 4.1(c)]{cameron2002transitive}, while Part (\ref{i2}) is \cite[Theorem 5.1]{cameron2002transitive}.
\end{proof}

We also need some theory concerning graph products to provide the machinery for our proof of Corollary \ref{cor:graphs}.
Let $\Gamma_1=(V_1,E_1)$ and $\Gamma_2=(V_2,E_2)$ be two simple graphs, i.e.~without loops or multiple edges. We recall two well-known graph products. The \emph{Cartesian product} $\Gamma_1 \square \Gamma_2$ and \emph{strong product} $\Gamma_1 \boxtimes \Gamma_2$ each have vertex set
\[
V(\Gamma_1 \square \Gamma_2) = V(\Gamma_1 \boxtimes \Gamma_2) = \{(v,w) \mid v\in V(\Gamma_1), w \in V(\Gamma_2)\}.
\]
Two distinct vertices $(v,w)$ and $(v',w')$ are adjacent in the Cartesian product if and only if either $v=v'$ and $w$ is adjacent to $w'$ in $\Gamma_2$, or $w = w'$ and $v$ is adjacent to $v'$ in $\Gamma_1$. They are adjacent in the strong product if and only if $v$ and $v'$ are either equal or adjacent in $\Gamma_1$ and $w$ and $w'$ are either equal or adjacent in $\Gamma_2$. Hence $\Gamma_1 \square \Gamma_2$ is a subgraph of $\Gamma_1 \boxtimes \Gamma_2$.

A graph is called \emph{prime} with respect to one of these graph products if it is non-trivial and cannot be written as the product of two non-trivial graphs. If a graph is non-trivial and not prime, we say that it is \emph{decomposable} with respect to that product. We further say that two graphs are \emph{relatively prime} with respect to a graph product if there is no non-trivial graph that (up to isomorphism) is a factor of both of them. We now prove some properties of graphs constructed from Cartesian and strong products.

\begin{proposition}
\label{prop:prod_props}
Let $\Gamma_1=(V_1,E_1)$ and $\Gamma_2=(V_2,E_2)$ be simple graphs on at least two vertices. Then the following statements hold.
\begin{enumerate}
\item \label{p1} The products $\Gamma_1 \square \Gamma_2$ and $\Gamma_1 \boxtimes \Gamma_2$ are connected if and only if $\Gamma_1$ and $\Gamma_2$ are, and vertex-transitive if and only if $\Gamma_1$ and $\Gamma_2$ are.
%\item \label{p0} Either $\Gamma_1 \boxtimes \Gamma_2$ is disconnected, or it has a complete subgraph with at least four vertices.
    \item \label{p2} If, for each $i \in \{1,2\}$, no two vertices in $\Gamma_i$ have the same closed neighbourhoods, then no two vertices in $\Gamma_1 \square  \Gamma_2$ have the same closed neighbourhoods. Separately, if $\Gamma_1$ and $\Gamma_2$ are connected and relatively prime with respect to the Cartesian product, then 
    $\Gamma_1 \square  \Gamma_2$ admits semiregular automorphisms only if at least one of $\Gamma_1$ and $\Gamma_2$ do.
    \item \label{p3} If, for each $i \in \{1,2\}$, no two vertices in $\Gamma_i$ have the same closed neighbourhoods, then no two vertices in $\Gamma_1 \boxtimes  \Gamma_2$ have the same closed neighbourhoods. If, in addition to this property, $\Gamma_1$ and $\Gamma_2$ are connected and relatively prime with respect to the strong product, then $\Gamma_1 \boxtimes  \Gamma_2$ admits semiregular automorphisms only if at least one of $\Gamma_1$ and $\Gamma_2$ do.
\end{enumerate}
\end{proposition}

\begin{proof}
Part~(\ref{p1}) follows immediately from Corollaries 5.3 and 5.6 and Theorems 6.17 and 7.19 in \cite{MR2817074}. Next, we consider Part (\ref{p2}). The statement about closed neighbourhoods in $\Gamma_1 \square \Gamma_2$ is immediate from the definition of the Cartesian product. Suppose now that $\Gamma_1$ and $\Gamma_2$ are connected and relatively prime with respect to the Cartesian product. Then \cite[Corollary 6.12]{MR2817074} yields $\mathrm{Aut}(\Gamma_1 \square  \Gamma_2) = \mathrm{Aut}(\Gamma_1)\times\mathrm{Aut}(\Gamma_2) $. Hence if $\Gamma_1 \square \Gamma_2$ admits a semiregular automorphism, then by considering this automorphism's projection onto each of $\Gamma_1$ and $\Gamma_2$, we see that at least one of these two factors also admits such an automorphism. The proof of Part (\ref{p3}) is similar, instead using \cite[Lemma 7.2, Corollary 7.17]{MR2817074}.
\end{proof}
We now provide the key result that will enable us to construct our desired family of graphs.
\begin{proposition}[{\cite[Exercise 7.6]{MR2817074}}]
\label{prop:pingpong}
Every connected graph that is decomposable with respect to the Cartesian product is prime with respect to the strong product and vice versa.
\end{proposition}

\section{Constructing elusive groups via non-split extensions}
\label{sec:nonsplit}

In this section, we formalise and generalise methods used by Chen et al.~\cite{chen2026elusive} to construct elusive groups via non-split extensions. We then discuss how these constructions apply to the groups discussed in Theorem~\ref{thm:main} and Remark~\ref{rem:otheractions}.

Given a set $\pi$ of primes, we will say that a transitive permutation group $G$ is \emph{$\pi$-elusive} if none of its derangements have orders in $\pi$, and \emph{$\pi'$-elusive} if each derangement $g$ of $G$ of prime order has $|g| \in \pi$ (this is a slight generalisation of notation used by Chen et al.~). Additionally, for an extension $N.G$ (with $N$ specified), let $\overline{\phantom{Y}}\!: X \to G$ be the natural projection, so that $\overline{Y} \cong NY/N$ for all $Y \le X$.

Recall that a non-split extension $\nonsplit{N}{G}$, corresponding to a specific action of a group $G$ on an elementary abelian group $N$ (so that $N$ is a specific $G$-module), exists if and only if the corresponding second cohomology group $\mathrm{H}^2(G,N)$ has positive dimension. By the Schur--Zassenhaus Theorem, this can only be the case if the prime dividing $|N|$ also divides $|G|$.

\begin{proposition}
\label{prop:const}
Let $G$ be a finite group, $H$ a core-free subgroup of $G$, and $N$ an elementary abelian group on which $G$ acts irreducibly, such that $\mathrm{H}^2(G,N)$ has positive dimension. Additionally, let $p$ be the prime dividing $|N|$, let $X$ be a non-split extension $\nonsplit{N}{G}$, and let $L$ be the preimage of $H$ under the projection $\overline{\phantom{R}}$. Suppose furthermore that the following conditions all hold.
\begin{enumerate}
    \item[(a)] In its action on the set $[G:H]$ of right cosets of $H$, the group $G$ is $\pi$-elusive for some set $\pi$ of primes.
    \item[(b)] The order of $G$ is not divisible by $p^2$.
    \item[(c)] The second cohomology group $\mathrm{H}^2(H,N)$ \textup{(}corresponding to the restriction to $H$ of the $G$-module $N$\textup{)} has dimension $0$.
\end{enumerate}
Then there is a one-to-one correspondence between:
\begin{enumerate}
\item subgroups $Y$ of $X$ \textup{(}up to $L$-conjugacy\textup{)} with $\overline{Y} = H$, such that $X$ is faithful and $(\pi \cup \{p\})$-elusive in its action on $[X:Y]$; and
\item proper $H$-submodules $M$ of $N$ satisfying $\bigcup_{g \in G}M^g = N$,
\end{enumerate}
given by $Y \cap N = M$.
\end{proposition}

\begin{proof}
Since $N \le L$ and $L/N \cong H$, the group $L$ is an extension of $N$ by $H$. Additionally, $\mathrm{H}^2(H,N)$ has dimension $0$, and so $L = N\mathbin{:}H$ (with $H$ identified with the appropriate subgroup of $L$). As each subgroup $Y$ of $L$ with $\overline{Y} = H$ satisfies $Y \cap N \trianglelefteq Y$, it follows that (up to $L$-conjugacy) these subgroups are precisely $M\mathbin{:}H$ for the $H$-submodules $M$ of $N$. (Note that if $M = N$, then $Y = L$ contains $N$ and so is not core-free in $X$.)

Now, fix a proper $H$-submodule $M$ of $N$, and let $Y:=M\mathbin{:}H$. As $N$ lies in the kernel of the action of $X$ on $N$, the action of $G$ on $N$ induces the action of $X$ on $N$. Since $N$ is an irreducible $G$-module and $\overline{Y} = H$ is core-free in $G$, we deduce that $Y$ is core-free in $X$. From now on, we will consider $X$ as a permutation group acting faithfully and transitively on $[X:Y]$.

As $p$ is the only prime dividing $N$, \cite[Lemma 2.1]{chen2026elusive} shows that $X$ is $(\pi \setminus \{p\})$-elusive. Hence it remains to consider whether $X$ has a derangement of order $p$. Observe that this is the case if and only if an element of order $p$ in $X$ lies in no $X$-conjugate of $Y$. Since $G$ is not divisible by $p^2$, a theorem of Gasch\"utz (see \cite[Main Theorem I.17.4]{hupperttranslated}) implies that $X \setminus N$ has no elements of order $p$, by the reasoning given in the proof of \cite[Lemma 2.5]{chen2026elusive}. Thus $X$ contains no derangement of order $p$ if and only if $\bigcup_{x \in X}(Y \cap N)^x = N$. Since the action of $X$ on $Y \cap N = M$ is induced by the action of $G$, this is the case if and only if $\bigcup_{g \in G}M^g = N$, as required.
\end{proof}

We note that Proposition~\ref{prop:const} does not depend on the specific non-split extension of the $G$-module $N$ by $G$, and that the condition $\bigcup_{g \in G}M^g = N$ also appears in \cite[Theorem 3.1]{cameron2002transitive} on the construction of elusive groups of affine type. Additionally, the Schur--Zassenhaus Theorem implies that Condition (c) holds whenever $p$ does not divide $|H|$. Furthermore, $|\{M^g \mid g \in G\}| \le |G:H|$ for each $H$-submodule $M$ of $N$. Hence if $M$ and $N$ have dimension $m$ and $n$, respectively, over $\mathbb{F}_p$ and $|G:H|\,(p^m-1) < p^n-1$, then ${\bigcup_{g \in G}M^g \subsetneq N}$.

Suppose now that $G$ is $\pi'$-elusive for some set $\pi$ of primes, and that there exist a fixed subgroup $H$ of $G$ and irreducible $G$-modules $N_p$ for each prime $p \in \pi$ that satisfy the conditions of Proposition~\ref{prop:const}. If each $N_p$ contains a proper $H$-submodule $M_p$ with $\bigcup_{g \in G}M_p^g = N_p$, then the proposition implies that each non-split extension $\nonsplit{(\prod_{p \in \pi}N_p)}{G}$ (corresponding to the relevant actions of $G$ on each $N_p$) is elusive (when acting on the set of cosets of an appropriate subgroup that projects onto $H$). By \cite[Theorem 5.3]{cameron2002transitive} (see also the discussion in \cite[Section 5]{chen2026elusive}), a counterexample to the Polycirculant Conjecture constructed via this method must have $\dim(N_p) - \dim(M_p) \ge 2$ for all $p \in \pi$.

Our constructions of the elusive groups described in Theorem~\ref{thm:main} and Remark \ref{rem:otheractions} appeal to Proposition~\ref{prop:const}. Indeed, $G_1 := \mathrm{PSU}_3(3)$ has a unique (up to conjugacy) subgroup $H_1$ isomorphic to $C_3 \times S_3$, and a unique irreducible $6$-dimensional $G$-module $N_1$ over $\mathbb{F}_7$. Similarly, $G_2 := \mathrm{P}\Gamma\mathrm{U}_3(3)$ has a unique (up to conjugacy) subgroup $H_2$ isomorphic to $S_3 \times S_3$, and a unique irreducible $6$-dimensional $G$-module $N_2$ over $\mathbb{F}_7$ satisfying $\dim(\mathrm{H}^2(G,N)) > 0$. It is easy to show computationally that, for each $i \in \{1,2\}$, the group $G_i$ is $\{7\}'$-elusive on $[G_i:H_i]$, that $G_i$, $H_i$ and $N_i$ satisfy the hypotheses and conditions of Proposition~\ref{prop:const}, and that there is a unique non-split extension $\nonsplit{N_i}{G_i}$. The elusive groups mentioned in Theorem~\ref{thm:main} and Remark~\ref{rem:otheractions} correspond to the $H_i$-submodules $M_i$ of $N_i$ satisfying $\bigcup_{g \in G_i}M_i^g = N_i$ and $\dim(N_i) - \dim(M_i) = 2$. Since the $2$-closures of the groups in Remark~\ref{rem:otheractions} are not elusive, the necessary condition given at the end of the previous paragraph is not sufficient for the constructed groups to be counterexamples to the Polycirculant Conjecture.

We conclude this section by noting that Problem 5.2 in \cite{chen2026elusive} asks if a counterexample to the Polycirculant Conjecture can be constructed as a non-split extension $\nonsplit{N}{G}$ with point stabiliser $Y$, where $G$ is quasiprimitive and $\{p\}'$-elusive on $[G:\overline{Y}]$ for some prime $p$ and $N$ is an irreducible $G$-module over $\mathbb{F}_p$. Since the $\{7\}'$-elusive permutation group $G_1$ is simple, it is quasiprimitive. Therefore, the $2$-closed elusive group $X = \nonsplit{N_1}{G_1}$ from Theorem~\ref{thm:main} and Corollary~\ref{cor:2closed} provides a positive answer to this problem.

\section{Proofs of Theorem \ref{thm:main} and Corollaries~\ref{cor:2closed} and \ref{cor:graphs}}
\label{sec:prf}
We now embark on the proofs of Theorem \ref{thm:main} and Corollaries~\ref{cor:2closed} and \ref{cor:graphs}, which combine theoretical ideas and computations. Our computations are performed in {\sc Magma}~\cite{MR1484478} via the accompanying code \cite{code}, which requires approximately 35 minutes of CPU time with a 3.85 GHz processor and 15 GB of RAM (or just 13 minutes and 0.8 GB if only the computations necessary to prove Theorem~\ref{thm:main} and Corollary~\ref{cor:graphs} are performed). The main results verified by our code will be emphasised via numbered Computational Propositions. 

\subsection{Proof of Theorem \ref{thm:main} and Corollary \ref{cor:2closed}}

Let $G = \mathrm{PSU}_3(3)$ and $N = 7^6$. There is a unique non-split extension $X:=\nonsplit{N}{G}$ with $G$ acting irreducibly on $N$, and (up to conjugacy) $X$ has exactly three subgroups of order $43{,}218=2.3^2.7^4$ (each isomorphic to groups of shape $7^4\mathbin{:}(C_3\times S_3)$) whose images under the natural projection from $X$ to $G$ are equal to a subgroup $H \cong C_3 \times S_3$ (cf.~the discussion at the end of Section~\ref{sec:nonsplit}). Of these three subgroups of $X$, one has exactly nine normal subgroups, while the other two have $16$ (cf.~Remark~\ref{rem:otheractions}). Let $Y$ be either of the latter two subgroups, and consider the action of $X$ on the set $\Omega:=[X:Y]$ of right cosets of $Y$ in $X$. 

Since $G$ is simple and acts irreducibly on $N$, we observe that $N$ is the unique proper non-trivial normal subgroup of the non-split extension $X$. Thus $Y$ is core-free in $X$, and so the action of $X$ is faithful and transitive of degree $16{,}464 = 2^4.3.7^3$. Our Magma code verifies the following Computational Proposition directly. We can alternatively obtain this result (as in Section~\ref{sec:nonsplit}) by checking that $G$, $N$, $H$ and $M:=Y \cap N$ satisfy (a), (b), (c) and (2) of Proposition~\ref{prop:const}, with $\pi$ the set of all primes other than $7$, and then applying that proposition.

\begin{compprop}
\label{compprop1}
$X$ is elusive.
\end{compprop}

We also observe computationally that there is an $X$-invariant partition $\Sigma$ of $\Omega$ consisting of $56$ blocks of size $294$ (as well as three other non-trivial $X$-invariant partitions). Let $\Delta \in \Sigma$, let $X^\Sigma$ be the permutation group induced by $X$ on $\Sigma$, and let $X_\Delta^\Delta$ be the permutation group induced by $X_\Delta$ on $\Delta$. As $X$ acts transitively on $\Omega$, Lemma \ref{lem:part} shows that $\Sigma$ is a $\kappa(X)$-invariant partition, that $\kappa(X)^\Sigma \le \kappa(X^\Sigma)$, and that $\kappa(X)_\Delta^\Delta \le \kappa(X_\Delta^\Delta)$.

Now, by Theorem \ref{thm:imprim_wr}, $\kappa(X)$ is a subgroup (up to permutational isomorphism) of the group $\kappa(X)_\Delta^\Delta \wr \kappa(X)^\Sigma$, in its natural imprimitive action on $\Xi := \Delta \times \Sigma$. It follows from the previous paragraph that $X \le \kappa(X) \le \kappa(X_\Delta^\Delta) \wr \kappa(X^\Sigma)$. In fact, we check computationally that $\kappa(X_\Delta^\Delta) = X_\Delta^\Delta$, and so $X \le \kappa(X) \le X_\Delta^\Delta \wr \kappa(X^\Sigma)$. We will write $A:=X_\Delta^\Delta$, $B:=\kappa(X^\Sigma)$ and $W:= A \wr B$.

Note that $B$ is a subgroup of the symmetric group of degree $|\Sigma| = 56$, and so $W = (A_1 \times A_2 \times \cdots \times A_{56}) : B$, where $A_i \cong A$ for each $i$. Let $C:= A_1 \times A_2 \times \cdots \times A_{56}$. We will identify $\Sigma$ with the set $\{1,\ldots,56\}$. For $(a_1,\ldots,a_{56}) \in C$, $b \in B$, $\delta \in \Delta$ and $i \in \Sigma$, the image $(\delta,i)^{(a_1,\ldots,a_{56})}$ is equal to $(\delta^{a_i},i)$, while $(\delta,i)^b = (\delta,i^b)$. Writing $\Delta_i:=\Delta \times \{i\}$, it follows that $\Delta_i^{(a_1,\ldots,a_{56})b} = \Delta_{i^b}$. It is therefore clear that each $A_i$ acts transitively on $\Delta_i$, and trivially on $\Delta_j$ for each $i \ne j$. Furthermore, the partition $\{\Delta_1,\ldots,\Delta_{56}\}$ of $\Sigma$ is $W$-invariant, and hence $X$-invariant.

Let $S:=C \cap X$, let $D:=N_W(S)$, and let $\alpha \in \Xi$. Since $C \trianglelefteq W$, it is clear that $S \trianglelefteq X$, and so $X \le D$. (In fact, $S$ is the unique proper non-trivial normal subgroup $N$ of $X$). Straightforward computations show that $X_\alpha$ has exactly $14$ orbits on $\Xi$ (each of size $63$) that extend to larger orbits of $D_\alpha$ (each of size $441$). Of the $14$ corresponding $X$-orbitals of size $63 \cdot 16{,}464$, exactly eight are self-paired, and of these eight, only one is $D$-conjugate to a non-self-paired $X$-orbital. Let $\mathcal{U}$ be the set of the remaining seven self-paired $X$-orbitals, whose orbital graphs all have valency 63.

\begin{compprop}
\label{compprop2}
The group $X$ is the full automorphism group of the orbital graph corresponding to each orbital in $\mathcal{U}$, and these seven graphs are connected and pairwise non-isomorphic.
\end{compprop}

Thus we have proved Theorem~\ref{thm:main}, and it immediately follows that $X$ is $2$-closed. However, as promised at the end of Section \ref{sec:intro}, we will continue to prove this latter fact without relying on the above automorphism group calculation. Note also that we have not been able to compute the automorphism groups of the orbital graphs corresponding to the remaining seven $X$-orbitals that are not $D$-orbitals. (The orbital graph corresponding to any $X$-orbital that is also a $D$-orbital clearly has automorphism group containing $D$.)

\begin{compprop}
\label{compprop3}
For each transitive overgroup $M$ of $S$ in $D$ with $|X|$ dividing but not equal to $|M|$, the number of orbits on $\Xi$ of the point stabiliser $M_\alpha$ is less than the number of orbits of $X_\alpha$.
\end{compprop}

Observe that $D \cap \kappa(X)$ contains $X$, and therefore contains $S$; that the transitivity of $X$ implies the transitivity of $D \cap \kappa(X)$; and that $(D \cap \kappa(X))_\alpha$ has the same orbits on $\Xi$ as $X_\alpha$. We therefore deduce from Computational Proposition~\ref{compprop3} that $D \cap \kappa(X) = X$.

Next, fix $c \in C \cap \kappa(X)$, and write $c = (c_1,\ldots,c_{56})$, with $c_i \in A_i$ for each $i$. By Proposition \ref{prop:wielandt}, for each $\alpha, \beta \in \Xi$, there exists $x \in X$ such that $(\alpha,\beta)^c = (\alpha,\beta)^x$. Consider the case where $\alpha \in \Delta_i$ and $\beta \in \Delta_j$ for $i \ne j$. Since $C$ stabilises each of $\Delta_i$ and $\Delta_j$ setwise and $\{\Delta_1,\ldots,\Delta_{56}\}$ is an $X$-invariant partition, $x$ lies in the intersection $T_{ij}$ of the setwise stabilisers $X_{\Delta_i}$ and $X_{\Delta_j}$. Therefore (cf.~the proof of \cite[Lemma 3.5(2)]{totally2closed}), for all $(\alpha,\beta) \in \Delta_i \times \Delta_j$, the pair $(\alpha,\beta)^c = (\alpha^{c_i},\beta^{c_j})$ lies in the orbit $(\alpha,\beta)^{T_{ij}}$. Equivalently, letting $A_{ij}:= \langle A_i, A_j \rangle$ and letting $\mathcal{O}_{ij}$ be the set of orbits of $T_{ij}$ on $\Delta_i \times \Delta_j$, the tuple $(1,\ldots,1,c_i,1,\ldots,1,c_j,1,\ldots,1)$ lies in the intersection $\bigcap_{O \in \mathcal{O}_{ij}} (A_{ij})_O$ of setwise stabilisers. Therefore, an element $c \in C$ lies in $\kappa(X)$ only if this condition holds for all distinct $i$ and $j$. (It does not take much more work to show that this is also a sufficient condition, but we will not need this fact.)

For each $i \in \{1,\ldots,56\}$, let $F_i$ be the subgroup of $A_i$ consisting of all elements $c_i$ such that, for all $j \ne i$, there exists $c_j \in A_j$ such that $\bigcap_{O \in \mathcal{O}_{ij}}(A_{ij})_O$ contains $(1,\ldots,1,c_i,1,\ldots,1,c_j,1,\ldots,1)$. Then $C \cap \kappa(X) \le \langle F_1,\ldots,F_{56}\rangle$.

\begin{compprop}
\label{compprop4}
There exists $k \in \{1,\ldots,56\}$ such that $F_k \le D$.
\end{compprop}

Let $x \in X$. Then $(T_{ij})^x = T_{i^xj^x}$ and $(A_i)^x = A_{i^x}$ for each $i$ and $j$, and so $(F_i)^x = F_{i^x}$. Since $X$ is transitive on $\Xi$, and since $D$ contains $\langle F_k, X \rangle$ by Computational Proposition~\ref{compprop4}, it follows that $D$ contains $\langle F_1,\ldots,F_{56}\rangle$, which contains $C \cap \kappa(X)$. Thus $C \cap \kappa(X) \le D \cap \kappa(X) = X$, and so the normal subgroup $C \cap \kappa(X)$ of $\kappa(X)$ is equal to $C \cap X = S$. Therefore, $\kappa(X) \le N_W(S) = D$. Again using the fact that $D \cap \kappa(X) = X$, we conclude that $\kappa(X) = X$, i.e.~$X$ is $2$-closed. Combining this fact with Computational Proposition~\ref{compprop1}, we see that $X$ is a counterexample to the Polycirculant Conjecture.

Finally, Proposition \ref{prop:dir_prod} shows that, for each integer $k \ge 1$, every group constructed from $k$ copies of $X$ via direct products and/or wreath products has a $2$-closed elusive action of degree $16{,}464^k$ on the set $\Omega^k$. Each of these groups is a counterexample to the Polycirculant Conjecture. This completes the proof of Corollary~\ref{cor:2closed}.

\subsection{Proof of Corollary \ref{cor:graphs}} 
Let $\mathcal{T}$ be the set of seven pairwise non-isomorphic orbital graphs corresponding to the set $\mathcal{U}$ of orbitals from Computational Proposition~\ref{compprop2}. Additionally, let $(\Gamma_i \mid i \in \mathbb{N})$ be any infinite sequence of graphs with $\Gamma_i \in \mathcal{T}$ for each $i$ and $\Gamma_1 \ne \Gamma_2$.
We define an associated infinite family $\mathcal{G} := (\mathcal{G}_i \mid i \in \mathbb{N})$ of graphs by $\mathcal{G}_1 := \Gamma_1$ and, for each $i \in \mathbb{N}$,
\[
\mathcal{G}_{2i} := \mathcal{G}_{2i-1} \square \Gamma_{2i}, \qquad \textrm{and} \qquad \mathcal{G}_{2i+1} := \mathcal{G}_{2i}\boxtimes \Gamma_{2i+1}.
\]
We claim that each member of $\mathcal{G}$ is a connected, vertex-transitive graph admitting no semiregular automorphism. Our computations reveal that, in each graph in $\mathcal{T}$, no two vertices have the same closed neigbourhoods. Moreover, \cite[Theorem 6.13]{MR2817074} shows that if any graph $\Gamma \in \mathcal{T}$ is not prime with respect to the Cartesian product, then its automorphism group is a direct product of wreath products. By \cite[Lemma 7.2, Theorem 7.18]{MR2817074} and the above closed neighbourhood property, the same is true for the strong product. However, Computational Proposition \ref{compprop2} shows that $X$ is the full automorphism group of $\Gamma$. The unique proper non-trivial normal subgroup of $X$ is $N \cong 7^6$, which has no complement in $X$. Hence $X$ cannot be written as a direct product, nor as a single wreath product. Therefore, $\Gamma$ must be prime with respect to the Cartesian and strong products.

We now proceed via induction. Suppose that, for some $i \in \mathbb{N}$, the graph $\mathcal{G}_i$ is connected and vertex-transitive, that no two of its vertices have equal closed neighbourhoods, and that it does not admit any semiregular automorphism. We immediately deduce from Proposition~\ref{prop:prod_props} that $\mathcal{G}_{i+1}$ is connected and vertex-transitive, and that no two of its vertices have equal closed neighbourhoods. By applying the previous paragraph when $i = 1$, or Proposition~\ref{prop:pingpong} and the definition of $\mathcal{G}_{i}$ when $i > 1$, we see that $\mathcal{G}_{i}$ is prime with respect to the product used to construct $\mathcal{G}_{i+1}$, and also relatively prime to $\Gamma_{i+1}$. As each graph in $\mathcal{T}$ admits no semiregular automorphism, it follows from Proposition \ref{prop:prod_props}(\ref{p2})--(\ref{p3}) and our inductive hypothesis that $\mathcal{G}_{i+1}$ also admits no semiregular automorphism. Using Proposition~\ref{prop:pingpong} and the fact that swapping $\Gamma_1$ and $\Gamma_2$ does not change the graphs in $\mathcal{G}$, it is easy to see that, for each integer $n \ge 2$, there is a two-to-one correspondence between the choices of $\Gamma_1,\ldots,\Gamma_n$ and the resulting graphs $\mathcal{G}_1,\ldots,\mathcal{G}_n$ (up to isomorphism). Since $|V(\mathcal{G}_i)| = 16{,}464^i$ for each $i\geq 1$, we obtain Corollary~\ref{cor:graphs}.

\begin{remark}
    Let $\Lambda_1$ and $\Lambda_2$ be two vertex-transitive (or, more generally, regular) graphs. By the definitions of their Cartesian and strong products, the valency $\val(\Lambda_1 \square \Lambda_2)$ of $\Lambda_1 \square \Lambda_2$ is equal to $\val(\Lambda_1)+\val(\Lambda_2)$, while ${\val(\Lambda_1 \boxtimes \Lambda_2)+1} = (\val(\Lambda_1)+1)(\val(\Lambda_2)+1)$. As $\val(\Gamma)$ has a constant value of $63$ for all $\Gamma \in \mathcal{T}$, we deduce that swapping the Cartesian and strong products in the definition of $\mathcal{G}$ results in a second countable set of infinite families of graphs giving positive answers to the Semiregularity Problem, with no overlap (up to isomorphism) other than the graphs in $\mathcal{T}$. It is additionally clear that, for each graph $\Gamma$ giving a positive answer to this problem, the complement of $\Gamma$ also gives a positive answer.
\end{remark}

\bibliographystyle{abbrv}
\bibliography{arXiv_bib}
\end{document}